\documentclass[10pt]{article}
\usepackage[english]{babel}

\usepackage{dsfont,amssymb,amsmath,amsthm,xcolor,makeidx,comment,graphicx}

\usepackage{etex}

\usepackage{tikz-cd}

\usepackage[noblocks,affil-it]{authblk}

\let\dim\relax 
\DeclareMathOperator{\dim}{dim}

\DeclareMathOperator{\lcm}{lcm} 
 
 \DeclareMathOperator{\Aut}{Aut}

\DeclareMathOperator{\vcd}{vcd}

\newcommand{\gen}[1]{\left\langle#1\right\rangle} 
\newcommand{\then}{\Rightarrow} 

\newcommand{\Iff}{\Longleftrightarrow} \newcommand{\ds}{\displaystyle}

 \newcommand{\zz}{\mathbb{Z}}

\let\ov\overline

\newtheorem{lemma}{Lemma}[section]
\newtheorem{theorem}[lemma]{Theorem}
\newtheorem{proposition}[lemma]{Proposition}
\newtheorem{corollary}[lemma]{Corollary}

\newcommand{\warning}[1]{\textcolor{red}{#1}}

 \DeclareMathOperator{\U}{{\cal U}}

\title{The cohomology ring of certain families of periodic virtually cyclic groups}
\date{\today}

\author{S\'ergio Tadao Martins}
\affil{Universidade Federal de Santa Catarina}
\author{Daciberg Lima Gon\c calves}
\affil{Universidade de S\~ao Paulo}
\author{M\'arcio de Jesus Soares}
\affil{Universidade Federal de S\~ao Carlos}

\begin{document}

\pagestyle{myheadings}
\markright{Cohomology  of virtually cyclic groups}

\maketitle

\begin{abstract}
Let $G$ be a virtually cyclic of the form $(\zz_a\rtimes \zz_b)\rtimes
\zz$ or $[\zz_a\rtimes(\zz_b\times Q_{2^i})]\rtimes \zz$. We compute
the integral cohomology ring of $G$, and then obtain the periodicity
of the Farell cohomology of these groups.

\end{abstract}

\section{Introduction}\label{intro}

Virtually cyclic groups can be considered the simplest family of
groups which contains both the infinite groups and the finite ones.
They play an important role in several subjects, like the Fibered
Isomorphism conjecture by Farrel and Jones, see \cite{FarrellJones},
and the study of space forms of infinite discrete groups. The space
forms for finite groups have been widely studied and, in the similar
problem for infinite discrete groups, the simplest case to consider is
the one where the groups are virtually cyclic. The structure of the
cohomology ring of these groups is an interesting question in its own
right and very useful for the study of the space forms. In more
detail, consider, for the infinite virtually cyclic groups, the
problem of deciding which ones act on homotopy spheres, as well the
classification of the homotopy type of the orbit spaces.

For results about space forms, mostly for finite groups and a few
infinite groups, see for example \cite{DacibergMarek1},
\cite{DacibergMarek2}, \cite{DacibergMarek3}, \cite{DacibergMarek4},
\cite{DacibergMarek5} \cite{DacibergMarek6}, \cite{GG7}, \cite{GG8}
and \cite{GG0}. In particular, for the special question of which
virtually cyclic groups act on homotopy spheres, the answer has been
given in \cite{GG0}. For the classification of the homotopy type of
the orbit spaces the cohomology ring of the group plays an important
role.

The goal of this work is to describe the ring structure of two
families of virtually cyclic groups, namely $(\zz_a\rtimes
\zz_b)\rtimes\zz$ and $[\zz_a\rtimes(\zz_b\times Q_{2^i})]\rtimes
\zz$.

The main  results of this work are
Theorems~\ref{theorem:zazbzring}, \ref{theorem:ggroups} and
\ref{theorem:gring}. Besides describing the cohomology ring structure
of these groups with integral coefficients, we also determine
explicitly the cohomology class which determine the periodicity (in
the sense of~\cite[Section X.6]{Brown82}).

This work contains 3 sections besides the introduction. In
section~\ref{prel} we present the basic tools about the
Lyndon-Hochschild-Serre spectral sequence which is going to be used
later. In section~\ref{zazb} we use section~\ref{prel} for our first
family of groups to obtain the main result which is
Theorem~\ref{theorem:zazbzring}. Section~\ref{zazbq2i} we use
section~\ref{prel} for our second family of groups to obtain the main
results which are Theorems~\ref{theorem:ggroups} and
\ref{theorem:gring}.

This project is in part sponsered by by FAPESP --- Funda\c c\~ao de Amparo
a Pesquisa do Estado de S\~ao Paulo, Projetos Tem\'aticos Topologia
Alg\'ebrica, Geom\'etrica e Differencial --- 2008/57607-6 and
2012/24454-8, and proccess {2013/07510-4}.

\section{Preliminaries}\label{prel}

Given an exact sequence of groups
\begin{equation}\label{eq:basicextension}
\begin{tikzcd}
1 \arrow{r} & H \arrow{r} & G \arrow{r} & Q \arrow{r} & 1,
\end{tikzcd}
\end{equation}
the Lyndon-Hochschild-Serre spectral sequence in cohomology is such
that
\[
E_2^{p,q} = H^p(Q;H^q(H;M)) \Longrightarrow H^{p+q}(G;M)
\]
for any $G$-module $M$. In this paper our main interest is in groups
of the form $F\rtimes_\theta \zz$, in which case the above spectral
sequence gives us the following result for $M=\zz$:

\begin{lemma}
For each positive integer $n$, there is an exact sequence
\[
\begin{tikzcd}
0 \arrow{r} &  H^{n-1}(F;\zz)_{\zz} \arrow{r} & H^n(F\rtimes_{\theta}\zz;\zz) \arrow{r} & H^n(F;\zz)^{\zz} \arrow{r} & 0.
\end{tikzcd}
\]
\end{lemma}
\begin{proof}
Since $H^n(\zz;M) = 0$ if $n\ge 2$, the spectral sequence associated
to the extension
\[
\begin{tikzcd}
  \phantom{,}1 \arrow{r} & F \arrow{r} & F\rtimes_\theta\zz \arrow{r} & \zz \arrow{r} & 1,
\end{tikzcd}
\]
is such that $E_2^{p,q} = 0$ for $p\ge 2$, hence $E_2^{p,q} =
E_\infty^{p,q}$ for all $p$, $q$ and we get for each $n\ge 1$ the
exact sequence
\[
\begin{tikzcd}
  \phantom{.}0 \arrow{r} & E_2^{1,n-1} \arrow{r} & H^n(F\rtimes_{\theta}\zz;\zz) \arrow{r} & E_2^{0,n} \arrow{r} & 0.
\end{tikzcd}
\]
On the other hand,
\begin{align*}
  E_2^{0,n} &= H^0(\zz;H^{n}(F;\zz))=H^{n}(F;\zz)^{\zz} \quad\text{and}\\
  E_2^{1,n-1} &= H^1(\zz;H^{n-1}(F;\zz))=H^{n-1}(F;\zz)_{\zz},
\end{align*}
so the result follows.
\end{proof}

\begin{corollary}\label{cor:FZ}
If $F$ is a finite group with periodic cohomology, then
$H^n(F\rtimes_{\theta}\zz;\zz)\cong H^n(F;\zz)^{\zz}$ for $n$ even and
$H^n(F\rtimes_{\theta}\zz;\zz) \cong H^{n-1}(F;\zz)_{\zz}$ for $n$
odd.
\end{corollary}
\begin{proof}
The result follows from the above Lemma and the fact that a finite
group with periodic cohomology has trivial odd dimensional integral
cohomology groups \cite[Section VI.9, exercise 4, page 159]{Brown82}.
\end{proof}

\bigskip

The Lyndon-Hochschild-Serre spectral sequence can be constructed in a
purely algebraic manner, but it can also be given a topological
interpretation: the extension~(\ref{eq:basicextension}) gives rise to
a fibration sequence
\[
\begin{tikzcd}
K(H,1) \arrow{r} & K(G,1) \arrow{r} & K(Q,1)
\end{tikzcd}
\]
and the resulting Leray-Serre spectral sequence is exactly the
Lyndon-Hochs\-child-Serre sequence (see~\cite{milgrameadem}). Theorem
5.2 of~\cite{McCleary} can now be applied to show that, on the
Lyndon-Hochschild-Serre spectral sequence, the cup product $\smile$
and the product $\cdot_2$ on $E_2$ are related by
\begin{equation}\label{eq:cupE2}
u\cdot_2 v = (-1)^{pq'}(u\smile v)
\end{equation}
for $u\in E_2^{p,q}$ and $v\in E_2^{p',q'}$. Depending on the
properties of $E_2$, this equation may be enough to recover the
cohomology ring $H^*(G;\zz)$ from the Lyndon-Hochschild-Serre spectral
sequence in cohomology. For instance, if $F$ is a finite group with
periodic cohomology, the Lyndon-Hochschild-Serre with $\zz$
coefficients applied to the extension
\[
\begin{tikzcd}
1 \arrow{r} & F \arrow{r} & F\rtimes_\theta\zz \arrow{r} & \zz
\arrow{r} & 1
\end{tikzcd}
\]
is such that $E_2^{p,q}=0$ if $p\ge 2$ or if $q$ is odd, so $E_2^{p,q}
= E_\infty^{p,q}$ and the product on $E_2$ gives us the cup product in
$H^*(F\rtimes_\theta\zz;\zz)$.

\bigskip
Now we make some remarks about the periodicity of the cohomology of
$(F\rtimes_{\theta}\zz)$: if $F$ is a finite group with periodic
cohomology, then $F\rtimes_{\theta}\zz$ has periodic Farrell
cohomology. This follows from~\cite[Theorem X.6.7, item
  (iii)]{Brown82}, since a finite subgroup of $F\rtimes_{\theta}\zz$
is actually a subgroup of $F\times \{0\}\cong F$. Moreover,
$\vcd(F\rtimes_{\theta}\zz)=1$ implies
$\hat{H}^i(F\rtimes_{\theta}\zz;M) \cong H^i(F\rtimes_{\theta}\zz;M)$
for all $i\ge 2$ and every $(F\rtimes_{\theta}\zz)$-module $M$, where
$\hat{H}^*$ denotes the Farrell cohomology. Finally, we know that the
cup product in the Farrell cohomology is compatible with the regular
cup product~\cite[Section X.3]{Brown82}, hence we have the following
result:

\begin{theorem}
Let $F$ be a finite group with periodic cohomology. If $d$ is a
positive integer and $u\in H^d(F\rtimes_{\theta}\zz;\zz)$ is a
cohomology class such that
\[
(u\smile\underline{\phantom{M}})\colon H^i(F\rtimes_{\theta}\zz;M) \to
H^{i+d}(F\rtimes_{\theta}\zz;M)
\]
is an isomorphism for all $i\ge 2$ and every
$(F\rtimes_{\theta}\zz)$-module $M$, then
\[
(\hat{u}\smile\underline{\phantom{M}})\colon
\hat{H}^i(F\rtimes_{\theta}\zz;M) \to
\hat{H}^{i+d}(F\rtimes_{\theta}\zz;M)
\]
is an isomorphism for all $i\in\zz$ and every
$(F\rtimes_{\theta}\zz)$-module $M$, where $\hat{u}$ denotes the image
of $u$ under the canonical isomorphism $H^d(F\rtimes_{\theta}\zz;M)
\stackrel{\cong}{\to} \hat{H}^d(F\rtimes_{\theta}\zz;M)$.
\end{theorem}

\section{The cohomology ring of $(\zz_a\rtimes\zz_b)\rtimes \zz$}\label{zazb}

Let $a$ and $b$ be relatively prime integers and let $G$ be the group
$(\zz_a\rtimes_\alpha\zz_b)\rtimes_\theta\zz$, where $\alpha\colon
\zz_b\to \Aut(\zz_a)$ is given by $\alpha(1_b)(1_a) = r\cdot 1_a$ with
$\gcd(a,(r-1)b)=1$ and $r^b\equiv 1 \pmod{a}$, and $\theta\colon
\zz\to \Aut(\zz_a\rtimes_\alpha\zz_b)$ is given by $\theta(1)(1_a) =
c_a\cdot 1_a$ and $\theta(1)(1_b) = c\cdot 1_a+c_b\cdot
1_b$. See~\cite{DacibergMarek3} for more details on the integers $c$,
$c_a$ and $c_b$.


The Lyndon-Hochschild-Serre spectral sequence in cohomology (with
$\zz$ coefficients) associated to the exact sequence of groups
\[
\begin{tikzcd}
1 \arrow{r} & \zz_a \arrow{r} & (\zz_a\rtimes_\alpha\zz_b) \arrow{r} &
\zz_b \arrow{r} & 1
\end{tikzcd}
\]
is such that $E_2^{p,q} = H^p(\zz_b;H^q(\zz_a;\zz))$. The group
$\zz_b$ acts trivially on $H^0(\zz_a;\zz)$ and acts on
$H^2(\zz_a;\zz)$ by multiplication by $r$. Since the cohomology ring
$H^*(\zz_a;\zz)$ is given by
\[
\frac{\zz[\alpha_2]}{(a\alpha_2=0)}
\]
with $\dim(\alpha_2)=2$, it follows that $\zz_b$ acts on
$H^{2j}(\zz_a;\zz)$ by multiplication by $r^j$. Hence $E_2^{p,q}$ is
possibly non-null only if $p$ and $q$ are even and $pq=0$ and it
follows that

\[
E_\infty^{p,q} = E_2^{p,q} = \begin{cases}
\zz, & \text{if $p=q=0$,} \\
\zz_{\delta_i}, & \text{if $p=0$ and $q=2i>0$,} \\
\zz_b, & \text{if $q=0$ and $p>0$ is even,} \\
0, & \text{otherwise,}
\end{cases}
\]
where $\delta_i = \gcd(r^i-1,a)$. Since $a$ and $b$ are relatively
prime, so are $\delta_i$ and $b$, and we have proved the following:

\begin{proposition}\label{adit}
The integral cohomology groups of $(\zz_a\rtimes_\alpha \zz_b)$, with
$\alpha(1_b)=r\cdot 1_a$, are given by
\[
H^n(\zz_a\rtimes_\alpha\zz_b; \zz) = \begin{cases}
\zz, & \text{if $n=0$,}\\
\zz_{\delta_i}\oplus\zz_b \cong \zz_{\delta_ib}, & \text{if $n=2i > 0$,} \\
0, & \text{if $n$ is odd.}
\end{cases}
\]
\end{proposition}

We will now compute the cohomology groups $H^*(G;\zz)$, and in order
to do that we need some simple facts and some notation.

If a cyclic group $H$ acts on the finite cyclic group $\zz_k$, then it
is straightforward to check that $(\zz_k)^H\cong (\zz_k)_H$,
hence $H^n(G;\zz)\cong H^{n+1}(G;\zz)$ if $n\ge 2$ is even.

The action of $\zz$ on $H^0(\zz_a\rtimes_\alpha\zz_b;\zz)$ is trivial,
so $H^0(G;\zz) \cong H^1(G;\zz) \cong \zz$. Also, noticing that
$\Aut(\zz_{\delta_i}\oplus\zz_b)=\Aut(\zz_{\delta_i})\oplus\Aut(\zz_b)$,
if we write
\[
H^{2i}(\zz_a\rtimes_\alpha\zz_b; \zz) = H^0(\zz_b; H^{2i}(\zz_a; \zz)) \oplus
H^{2i}(\zz_b; H^0(\zz_a; \zz)) \cong \zz_{\delta_i}\oplus \zz_b,
\]
then the action of $1\in\zz$ on $H^0(\zz_b; H^{2i}(\zz_a; \zz))$ is
the multiplication by $c_a^i$ on $\zz_{\delta_i}$ and the action of
$1\in\zz$ on $H^{2i}(\zz_b; H^0(\zz_a; \zz))$ is the multiplication by
$c_b^i$ on $\zz_b$, where the integers $c_a$ and $c_b$ come from the
definition of $\theta\colon \zz\to \Aut(\zz_a\rtimes_\alpha\zz_b)$ as
\[
\theta(1)(1_a) = c_a\cdot 1_a \quad\text{and}\quad \theta(1)(1_b) =
c\cdot 1+c_b\cdot 1_b,
\]
as we mentioned in the beginning of this section. Therefore, if we
define $A_j = \gcd(c_a^j-1,\delta_i)$ for $i\equiv j \pmod{d}$ and
$B_j = \gcd(c_b^j-1, b)$, we have the following result:

\begin{theorem}
The integral cohomology groups of $G=(\zz_a\rtimes_\alpha \zz_b)
\rtimes_\theta\zz$ are given by
\[
H^n(G; \zz) = \begin{cases}
\zz, & \text{if $n=0$ or $n=1$}, \\
\zz_{A_j}\oplus \zz_{B_j}, & \text{if $n=2j$ or $n=2j+1$ (and $j>0$)}.
\end{cases}
\]
\end{theorem}

\begin{proof}  From   Corollary~\ref{cor:FZ} we have
\[
H^n(G;\zz) = \begin{cases}
H^n(\zz_a\rtimes_\alpha\zz_b; \zz)^\zz, & \text{if $n$ is even,} \\
H^{n-1}(\zz_a\rtimes_\alpha\zz_b; \zz)_\zz, & \text{if $n$ is odd.} \\
\end{cases}
\]
Using the Proposition \ref{adit} and the notation introduced above,
the result follows.
\end{proof}

In order to determine the multiplicative structure in $H^*(G;\zz)$, we
need representatives for the generating classes of each cohomology
group $H^n(G;\zz)$. For $n=0$, let us denote the generating class by
$1$. For $n=1$, a representative for the generating class of
\[
H^1(\zz; H^0(\zz_a\rtimes_\alpha\zz_b; \zz)) \cong \zz
\]
is the map $\eta\colon \zz[\zz_a\rtimes_\alpha\zz_b] \to \zz$ of
$\zz[\zz_a\rtimes_\alpha\zz_b]$-modules given by $\eta(1)=1$. More
precisely, if $\tilde\eta\colon \zz[\zz] \to
H^0(\zz_a\rtimes_\alpha\zz_b; \zz)$ is the $\zz[\zz]$-homomorphism
defined by $\tilde\eta(1) = [\eta]$, then $[\tilde\eta]$ generates
$H^1(\zz; H^0(\zz_a\rtimes_\alpha\zz_b; \zz))$ and $[\eta]$ is the
identity in $H^*(\zz_a\rtimes_\alpha\zz_b; \zz)$.

For $n\ge 2$, we need representatives for the generating classes of
the following groups (where $i>0$):
\begin{enumerate}
\item $H^0(\zz; H^0(\zz_b; H^{2i}(\zz_a; \zz))) = H^0(\zz_b;
  H^{2i}(\zz_a; \zz))^\zz$
\item $H^0(\zz; H^{21i}(\zz_b; H^0(\zz_a; \zz))) = H^{2i}(\zz_b;
  H^0(\zz_a; \zz))^\zz$
\item $H^1(\zz; H^0(\zz_b; H^{2i}(\zz_a; \zz))) = H^0(\zz_b;
  H^{2i}(\zz_a; \zz))_\zz$
\item $H^1(\zz; H^{2i}(\zz_b; H^0(\zz_a; \zz))) = H^{2i}(\zz_b;
  H^0(\zz_a; \zz))_\zz$
\end{enumerate}

We start with $H^0(\zz; H^0(\zz_b; H^{2i}(\zz_a; \zz))) = H^0(\zz_b;
H^{2i}(\zz_a; \zz))^\zz$. We know that $\zz_b$ acts on $H^{2i}(\zz_a;
\zz)$ by multiplication by $r^i$, so $H^0(\zz_b; H^{2i}(\zz_a;
\zz))^\zz \cong (\zz_{\delta_i})^\zz$, where
\begin{align*}
\zz_{\delta_i} &= \{\ov{x} \in \zz_a \::\: r^ix \equiv x \pmod{a}\} \\
&= \{\ov{x} \in \zz_a \::\: x\equiv 0 \pmod{a/\delta_i}\} \\
&= \{\ov{x} \in \zz_a \::\: x = (ka)/\delta_i, \: k\in\zz\}.
\end{align*}
Since $\zz$ acts on $\zz_{\delta_i}$ by multiplication by $c_a^i$, we
obtain
\[
(\zz_{\delta_i})^\zz = \{\ov{x} \in \zz_{\delta_i} \::\:
  (c_a^i-1)x\equiv 0 \pmod{a}\},
\]
and the fact that $x = (ka)/\delta_i$ yields the equation
$\dfrac{(c_a^i-1)ka}{\delta_i} \equiv 0 \pmod{a}$, which has the solution
\[
k \equiv 0 \left(\bmod \:\gcd\left(a,
\dfrac{a(c_a^i-1)}{\delta_i}\right)\right) \Iff k\equiv 0 \left(\bmod
\: \dfrac{a}{\delta_i}\gcd(\delta_i, c_a^i-1)\right).
\]
Hence a representative for the generating class of
\[
H^0(\zz; H^0(\zz_b; H^{2i}(\zz_a; \zz))) \cong (\zz_{\delta_i})^\zz
\]
is the map $u_a^i\colon \zz[\zz_a] \to \zz$ given by
\[
u_a^i(1) = \frac{a}{\delta_i}\cdot\frac{\delta_i}{\gcd(\delta_i,
  (c_a^i-1))} = \frac{a}{\gcd(\delta_i, c_a^i-1)}.
\]
More precisely, if $\varphi_a^i\colon \zz[\zz_b]\to H^{2i}(\zz_a,\zz)$
is the $\zz[\zz_b]$-homomorphism given by $\varphi_a^i(1) = [u_a^i]$
and if $\tilde{\varphi}_a^i\colon \zz[\zz] \to H^0(\zz_b,
H^{2i}(\zz_a,\zz))$ is the $\zz[\zz]$-homomorphism given by
$\tilde{\varphi}_a^i(1) = [\varphi_a^i]$, then $[\tilde{\varphi}_a^i]$
is a generating class for
\[
H^0(\zz; H^0(\zz_b; H^{2i}(\zz_a; \zz))) = H^0(\zz_b; H^{2i}(\zz_a;
\zz))^\zz \cong (\zz_{\delta_i})^\zz.
\]

Similarly, one representative for the generating class of
\[
H^0(\zz; H^{2i}(\zz_b; H^0(\zz_a; \zz))) = H^{2i}(\zz_b; H^0(\zz_a;
\zz))^\zz \cong (\zz_b)^\zz
\]
is the map $u_b^i\colon \zz[\zz_a] \to \zz$ given by
\[
u_b^i(1) = \frac{b}{\gcd(b,c_b^i-1)},
\]
and we define $\varphi_b^i\colon \zz[\zz_b]\to H^0(\zz_a;\zz)$ and
$\tilde{\varphi}_b^i\colon \zz[\zz] \to H^{2i}(\zz_b; H^0(\zz_a;\zz))$
in a way similar to $\varphi_a^i$ and $\tilde{\varphi}_a^i$.

The same reasoning applies to determine a representative for the
generating class of
\[
H^1(\zz; H^0(\zz_b; H^{2i}(\zz_a; \zz))) = H^0(\zz_b;
H^{2i}(\zz_a; \zz))_\zz \cong (\zz_{\delta_i})_\zz,
\]
this representative being the map $v_a^i\colon \zz_a\to \zz$ defined
by
\[
v_a^i(1) = \frac{a}{\delta_i}.
\]
Finally, a representative for the generating class of
\[
H^1(\zz; H^{2i}(\zz_b; H^0(\zz_a; \zz))) = H^{2i}(\zz_b; H^0(\zz_a;
\zz))_\zz = (\zz_b)_\zz
\]
is the map $v_b^i\colon \zz[\zz_a] \to \zz$ given by
\[
v_b^i(1) = 1.
\]
The maps $\psi_a^i$, $\tilde{\psi}_a^i$, $\psi_b^i$ and
$\tilde{\psi}_b^i$ are defined similarly to $\varphi_a^i$,
$\tilde{\varphi}_a^i$, $\varphi_b^i$ and $\tilde{\varphi}_b^i$.

\bigskip
The last necessary ingredients we need to calculate the cup products
in $H^*(G,\zz)$ are the diagonal approximations for the known free
resolutions of $\zz$ over $\zz[\zz]$ and over $\zz[\zz_m]$ for $m>1$.

For $\zz= \gen{s}$, a free resolution $F$ of $\zz$ over $\zz G$ is
given by
\[
\begin{tikzcd}
0 \arrow{r} & F_1 \arrow{r}{s-1} & F_0 \arrow{r}{\varepsilon} & \zz \arrow{r} & 0,
\end{tikzcd}
\]
where $F_1=F_0=\zz[\zz]$,%
%
and a direct verification shows that $\Delta_n\colon F_n \to (F\otimes
F)_n$ defined by
\begin{align}\label{eq:diagonalz}
\Delta_0 \colon F_0 &\to F_0\otimes F_0 \notag\\
\Delta_0(1) &= 1\otimes 1, \notag\\
\Delta_1 \colon F_1 &\to (F_1\otimes F_0)\oplus (F_0\otimes F_1) \notag\\
\Delta_1(1) &= \underbrace{1\otimes s}_{F_1\otimes F_0} + \underbrace{1\otimes 1}_{F_0\otimes F_1}
\end{align}
is a diagonal approximation for the resolution $F$.

For the cyclic group $\zz_m = \gen{t \mid t^m=1}$, a free resolution
$P$ of $\zz$ over $\zz[\zz_m]$ and a diagonal approximation for $P$
are found in~\cite{Brown82}. The resolution $P$ is given by
\[
\begin{tikzcd}
\cdots \arrow{r}{N} & \zz[\zz_m] \arrow{r}{t-1} & \zz[\zz_m] \arrow{r}{N}
& \zz[\zz_m] \arrow{r}{t-1} & \zz[\zz_m] \arrow{r}{\varepsilon} & \zz
\arrow{r} & 0,
\end{tikzcd}
\]
where $P_n = \zz[\zz_m]$ for all $n\ge 0$ and
$N=1+t+\cdots+t^{m-1}$. A diagonal approximation $\Delta$ for $P$ has
components $\Delta_{pq}\colon P_{p+q} \to P_p\otimes P_q$ given by
\begin{equation}\label{eq:diagonalzm}
\Delta_{pq}(1) =
\begin{cases}
1\otimes 1, & \text{if $p$ is even}, \\
1\otimes t, & \text{if $p$ is odd, $q$ is even}, \\
\ds\sum_{0\le i<j \le m-1} t^i\otimes t^j,  & \text{if $p$ and $q$ are odd}. \\
\end{cases}
\end{equation}

\bigskip
Now let $i,j>0$. We begin by considering the product
\begin{align*}
H^0(\zz; H^0(\zz_b; H^{2i}(\zz_a; \zz))) &\otimes H^0(\zz; H^0(\zz_b; H^{2j}(\zz_a; \zz))) \to \\
\to& H^0(\zz; H^0(\zz_b; H^{2i}(\zz_a; \zz))\otimes H^0(\zz_b; H^{2j}(\zz_a; \zz))) \to \\
\to& H^0(\zz; H^0(\zz_b; H^{2i}(\zz_a; \zz)\otimes H^{2j}(\zz_a; \zz))) \to \\
\to& H^0(\zz; H^0(\zz_b; H^{2(i+j)}(\zz_a; \zz))),
\end{align*}
where in the above composition we evaluate the cup products in
$H^*(\zz,\underline{\phantom{M}})$,
$H^*(\zz_b,\underline{\phantom{M}})$ and
$H^*(\zz_a,\underline{\phantom{M}})$, and make the identification
$\zz\otimes\zz \cong \zz$. Using the diagonal approximations
~(\ref{eq:diagonalz}) and~(\ref{eq:diagonalzm}), we obtain
\begin{align*}
(\tilde{\varphi}_a^i \smile\tilde{\varphi}_a^j)(1) &= \tilde{\varphi}_a^i(1) \otimes \tilde{\varphi}_a^j(1) = [\varphi_a^i]\otimes[\varphi_a^j], \\
(\varphi_a^i \smile \varphi_a^j)(1) &= \varphi_a^i(1)\otimes \varphi_a^j(1) = [u_a^i]\otimes [u_a^j], \\
(u_a^i\smile u_a^j)(1) &= u_a^i(1)\otimes u_a^j(1) \mapsto \frac{a^2}{\gcd(\delta_i, c_a^i-1)\gcd(\delta_j,c_a^j-1)},
\end{align*}
from where it follows that
\[
[\tilde{\varphi}_a^i]\smile[\tilde{\varphi}_a^j] = \frac{a\gcd(\delta_{i+j},c_a^{i+j}-1)}{\gcd(\delta_i, c_a^i-1)\gcd(\delta_j,c_a^j-1)}[\tilde{\varphi}_a^{i+j}].
\]

Similarly, the product
\begin{align*}
H^0(\zz; H^{2i}(\zz_b; H^0(\zz_a; \zz))) &\otimes H^0(\zz; H^{2j}(\zz_b; H^0(\zz_a; \zz))) \to \\
\to& H^0(\zz; H^{2i}(\zz_b; H^0(\zz_a; \zz))\otimes H^{2j}(\zz_b; H^0(\zz_a; \zz))) \to \\
\to& H^0(\zz; H^{2(i+j)}(\zz_b; H^0(\zz_a; \zz)\otimes H^0(\zz_a; \zz))) \to \\
\to& H^0(\zz; H^{2(i+j)}(\zz_b; H^0(\zz_a; \zz)))
\end{align*}
is such that
\[
[\tilde{\varphi}_b^i]\smile[\tilde{\varphi}_b^j] = \frac{b\gcd(b,c_b^{i+j}-1)}{\gcd(b, c_b^i-1)\gcd(b,c_b^j-1)}[\tilde{\varphi}_b^{i+j}].
\]

As to the product
\begin{align*}
H^0(\zz; H^0(\zz_b; H^{2i}(\zz_a; \zz))) &\otimes H^0(\zz; H^{2j}(\zz_b; H^0(\zz_a; \zz))) \to \\
\to& H^0(\zz; H^0(\zz_b; H^{2i}(\zz_a; \zz))\otimes H^{2j}(\zz_b; H^0(\zz_a; \zz))) \to \\
\to& H^0(\zz; H^{2j}(\zz_b; H^{2i}(\zz_a; \zz)\otimes H^0(\zz_a; \zz))) \to \\
\to& H^0(\zz; H^{2j}(\zz_b; H^{2i}(\zz_a; \zz))) = 0,
\end{align*}
we can only have
\[
[\tilde{\varphi}_a^i]\smile[\tilde{\varphi}_b^j] = 0.
\]
Similarly, the compositions
\begin{align*}
H^0(\zz; H^0(\zz_b; H^{2i}(\zz_a; \zz))) &\otimes H^1(\zz; H^{2j}(\zz_b; H^0(\zz_a; \zz))) \to \\
\to& H^1(\zz; H^0(\zz_b; H^{2i}(\zz_a; \zz)) \otimes H^{2j}(\zz_b; H^0(\zz_a; \zz))) \to \\
\to& H^1(\zz; H^{2j}(\zz_b; H^{2i}(\zz_a; \zz) \otimes H^0(\zz_a; \zz))) \to \\
\to& H^1(\zz; H^{2j}(\zz_b; H^{2i}(\zz_a; \zz))) = 0
\end{align*}
and
\begin{align*}
H^0(\zz; H^{2i}(\zz_b; H^0(\zz_a; \zz))) &\otimes H^1(\zz; H^0(\zz_b; H^{2j}(\zz_a; \zz))) \to \\
\to& H^1(\zz; H^{2i}(\zz_b; H^0(\zz_a; \zz)) \otimes H^0(\zz_b; H^{2j}(\zz_a; \zz))) \to \\
\to& H^1(\zz; H^{2i}(\zz_b; H^0(\zz_a; \zz) \otimes H^{2j}(\zz_a; \zz))) \to \\
\to& H^1(\zz; H^{2i}(\zz_b; H^{2j}(\zz_a; \zz))) = 0
\end{align*}
show us that
\[
[\tilde{\varphi}_a^i]\smile[\tilde{\psi}_b^j] = 0, \quad [\tilde{\varphi}_b^i]\smile[\tilde{\psi}_a^j] = 0.
\]

The product
\begin{align*}
H^0(\zz; H^0(\zz_b; H^{2i}(\zz_a; \zz))) &\otimes H^1(\zz; H^0(\zz_b; H^{2j}(\zz_a; \zz))) \to \\
\to& H^1(\zz; H^0(\zz_b; H^{2i}(\zz_a; \zz)) \otimes H^0(\zz_b; H^{2j}(\zz_a; \zz))) \to \\
\to& H^1(\zz; H^0(\zz_b; H^{2i}(\zz_a; \zz) \otimes H^{2j}(\zz_a; \zz))) \to \\
\to& H^1(\zz; H^0(\zz_b; H^{2(i+j)}(\zz_a; \zz)))
\end{align*}
is such that
\begin{align*}
(\tilde{\varphi}_a^i \smile\tilde{\psi}_a^j)(1) &= \tilde{\varphi}_a^i(1) \otimes \tilde{\psi}_a^j(1) = [\varphi_a^i]\otimes[\psi_a^j], \\
(\varphi_a^i \smile \psi_a^j)(1) &= \varphi_a^i(1)\otimes \psi_a^j(1) = [u_a^i]\otimes [v_a^j], \\
(u_a^i\smile v_a^j)(1) &= u_a^i(1)\otimes v_a^j(1) \mapsto \frac{a^2}{\delta_i\gcd(\delta_i, c_a^i-1)},
\end{align*}
hence
\[
[\tilde{\varphi}_a^i]\smile [\tilde{\psi}_a^j] = \frac{a\delta_{i+j}}{\delta_j\gcd(\delta_i,c_a^i-1)}[\tilde{\psi}_a^{i+j}].
\]

Similarly, the product
\begin{align*}
H^0(\zz; H^{2i}(\zz_b; H^0(\zz_a; \zz))) &\otimes H^1(\zz; H^{2j}(\zz_b; H^0(\zz_a; \zz))) \to \\
\to& H^1(\zz; H^{2i}(\zz_b; H^0(\zz_a; \zz)) \otimes H^{2j}(\zz_b; H^0(\zz_a; \zz))) \to \\
\to& H^1(\zz; H^{2(i+j)}(\zz_b; H^0(\zz_a; \zz) \otimes H^0(\zz_a; \zz))) \to \\
\to& H^1(\zz; H^{2(i+j)}(\zz_b; H^0(\zz_a; \zz)))
\end{align*}
is such that
\[
[\tilde{\varphi}_b^i]\smile [\tilde{\psi}_b^j] = \frac{b}{\gcd(b,c_b^i-1)}[\tilde{\psi}_b^{i+j}].
\]
Finally, observing that $[\eta]$ is the identity in
$H^*(\zz_a\rtimes_\alpha \zz_b;\zz)$, we also have
\[
[\tilde{\varphi}_a^i] \smile [\tilde{\eta}] = \frac{\delta_i}{\gcd(\delta_i,c_a^i-1)}[\tilde{\psi}_a^i]
\]
and
\[
[\tilde{\varphi}_b^i] \smile [\tilde{\eta}] = \frac{b}{\gcd(b,c_b^i-1)}[\tilde{\psi}_b^i].
\]

Therefore, we have the following result:

\begin{theorem}\label{theorem:zazbzring}
Let $G=(\zz_a\rtimes_\alpha \zz_b)\rtimes_\theta \zz$, where
$\gcd(a,b)=1$, and with $\alpha$ and $\theta$ defined by
\begin{align*}
\alpha\colon &\zz_b \to \Aut(\zz_a) \\
&\alpha(1_b)(1_a) = r\cdot 1_a
\end{align*}
and
\begin{align*}
\theta\colon &\zz \to \Aut(\zz_a\rtimes_\alpha \zz_b)\\
&\theta(1)(1_a) = c_a\cdot 1_a,\\
&\theta(1)(1_b) = c\cdot 1_a + c_b\cdot 1_b.
\end{align*}
The cohomology groups $H^*(G;\zz)$ are given by
\[
H^n(G;\zz) \cong \begin{cases}
\zz, & \text{if $n=0$ or $n=1$}, \\
\zz_{A_j}\oplus \zz_{B_j}, & \text{if $n=2j$ or $n=2j+1$ (and $j>0$)},
\end{cases}
\]
where $A_j = \gcd(c_a^j-1,\delta_j)$, $\delta_j = \gcd(r^j-1,a)$ and
$B_j = \gcd(c_b^i-1,b)$. Moreover, there is a generator
$[\tilde{\eta}]$ of $H^1(G,\zz)$, there are generators
$[\tilde{\varphi}_a^i]$ and $[\tilde{\varphi}_b^i]$ of $H^{2i}(G;\zz)$
for $i>0$ such that $\gen{[\tilde{\varphi}_a^i]} \cong \zz_{A_i}$ and
$\gen{[\tilde{\varphi}_b^i]} \cong \zz_{B_i}$, and there are
generators $[\tilde{\psi}_a^i]$ and $[\tilde{\psi}_b^i]$ of
$H^{2i+1}(G;\zz)$ for $i>0$ such that $\gen{[\tilde{\psi}_a^i]} \cong
\zz_{A_i}$ and $\gen{[\tilde{\psi}_b^i]} \cong \zz_{B_i}$ for which we
have
\begin{align*}
&[\tilde{\varphi}_a^i] \smile [\tilde{\eta}] = \frac{\delta_i}{\gcd(\delta_i,c_a^i-1)}[\tilde{\psi}_a^i], \\
&[\tilde{\varphi}_b^i] \smile [\tilde{\eta}] = \frac{b}{\gcd(b,c_b^i-1)}[\tilde{\psi}_b^i], \\
&[\tilde{\psi_a^i}] \smile [\tilde{\eta}] = 0, \\
&[\tilde{\psi_b^i}] \smile [\tilde{\eta}] = 0, \\
&[\tilde{\varphi}_a^i] \smile [\tilde{\varphi}_a^j] = \frac{a\gcd(\delta_{i+j},c_a^{i+j}-1)}{\gcd(\delta_i, c_a^i-1)\gcd(\delta_j,c_a^j-1)}[\tilde{\varphi}_a^{i+j}], \\
&[\tilde{\varphi}_b^i] \smile [\tilde{\varphi}_b^j] = \frac{b\gcd(b,c_b^{i+j}-1)}{\gcd(b, c_b^i-1)\gcd(b,c_b^j-1)}[\tilde{\varphi}_b^{i+j}], \\
&[\tilde{\varphi}_a^i] \smile [\tilde{\varphi}_b^j] = 0, \\
&[\tilde{\varphi}_a^i] \smile [\tilde{\psi}_a^j] = \frac{a\delta_{i+j}}{\delta_j\gcd(\delta_i,c_a^i-1)}[\tilde{\psi}_a^{i+j}], \\
&[\tilde{\varphi}_b^i] \smile [\tilde{\psi}_b^j] = \frac{b}{\gcd(b,c_b^i-1)}[\tilde{\psi}_b^{i+j}], \\
&[\tilde{\varphi}_a^i] \smile [\tilde{\psi}_b^j] = 0, \\
&[\tilde{\varphi}_b^i] \smile [\tilde{\psi}_a^j] = 0, \\
&[\tilde{\psi}_a^i] \smile [\tilde{\psi}_a^j] = 0, \\
&[\tilde{\psi}_a^i] \smile [\tilde{\psi}_b^j] = 0, \\
&[\tilde{\psi}_b^i] \smile [\tilde{\psi}_b^j] = 0.
\end{align*}
Finally, if we let $d$, $d_{c_a}$ and $d_{c_b}$ be the orders of $r$,
$c_a$ and $c_b$ in $\U(\zz_a)$, respectively, and $p =
\lcm(d,d_{c_a},d_{c_b})$, then
\[
([\tilde{\varphi}_a^p] + [\tilde{\varphi}_b^p]) \smile \underline{\phantom{M}}\colon H^n(G;\zz) \to H^{n+2p}(G;\zz)
\]
is an isomorphism for all $n\ge 2$.
\end{theorem}
\begin{proof}
We only need to prove the statement concerning the periodicity of
$H^*(G;\zz)$. Noticing that $r^p \equiv 1 \pmod{a}$, $c_a^p\equiv 1
\pmod{a}$ and $c_b^p \equiv 1 \pmod{b}$, we have
\begin{align*}
\delta_p &= a, \\
\delta_{p+j} &= \delta_j, \\
\gcd(\delta_{p+j}, c_a^{p+j}-1) &= \gcd(\delta_j, c_a^j-1), \\
\gcd(b, c_b^{p+j}-1) &= \gcd(b, c_b^j-1).
\end{align*}
Therefore, using the formulas we already have for the cup products
show us that, for all $n\ge 2$,
\begin{align*}
& ([\tilde{\varphi}_a^p] + [\tilde{\varphi}_b^p])\smile [\tilde{\varphi}_a^j] = [\tilde{\varphi}_a^{p+j}], \\
& ([\tilde{\varphi}_a^p] + [\tilde{\varphi}_b^p])\smile [\tilde{\varphi}_b^j] = [\tilde{\varphi}_b^{p+j}], \\
& ([\tilde{\varphi}_a^p] + [\tilde{\varphi}_b^p])\smile [\tilde{\psi}_a^j] = [\tilde{\psi}_a^{p+j}], \\
& ([\tilde{\varphi}_a^p] + [\tilde{\varphi}_b^p])\smile [\tilde{\psi}_b^j] = [\tilde{\psi}_b^{p+j}], \\
\end{align*}
and that proves that
\[
([\tilde{\varphi}_a^p] + [\tilde{\varphi}_b^p]) \smile \underline{\phantom{M}}\colon H^n(G;\zz) \cong H^{n+2p}(G;\zz)
\]
is an isomorphism for $n\ge 2$.
\end{proof}

\begin{corollary}
If $[\tilde{\varphi}_a^p] + [\tilde{\varphi}_b^p] \in H^{2p}(G;\zz)$
is the cohomology class described in the previous theorem, then
\[
(\widehat{[\tilde{\varphi}_a^p] + [\tilde{\varphi}_b^p]}) \smile \underline{\phantom{M}}\colon \hat{H}^n(G;\zz) \cong \hat{H}^{n+2p}(G;\zz)
\]
is an isomorphism for all $n\in\zz$.
\end{corollary}

\section{The cohomology ring of $[\zz_a\rtimes(\zz_b\times Q_{2^i})]\rtimes \zz$}\label{zazbq2i}

Let $i\ge 3$ be an integer and $Q_{2^i} = \langle x,y \:\mid\:
x^{2^{i-2}} = y^2,\: xyx=y\rangle$. The group $G =
[\zz_a\rtimes(\zz_b\times Q_{2^i})]\rtimes \zz$ for $\gcd(a,b) =
\gcd(ab,2) =1$ can be obtained as a sequence of extensions, and we
will follow through this sequence in order to calculate its integral
cohomology ring.

First, we note that the integral cohomology ring of the groups
$Q_{2^i}$ (actually, $Q_{4n}$) has been determined
in~\cite{TomodaPeter}:

\begin{theorem}[Tomoda and Zvengrowski,~\cite{TomodaPeter}]\label{theorem:TomodaPeter}
The cohomology ring $H^*(Q_{2^i};\zz)$ has the following presentation
for $i\ge 3$:
\[
  H^*(Q_{2^i}; \zz) = \begin{cases}
    \dfrac{\zz[\gamma_2,\gamma_2',\delta_4]}
          {\begin{pmatrix}
              2\gamma_2 = 2\gamma_2' = 2^i\delta_4 = 0\\
              \gamma_2^2 = 0, \gamma_2'^2 = \gamma_2\gamma_2' = 2^{i-1}\delta_4
          \end{pmatrix}}, & \text{if $i>3$;}\\
    \dfrac{\zz[\gamma_2,\gamma_2',\delta_4]}
          {\begin{pmatrix}
              2\gamma_2 = 2\gamma_2' = 8\delta_4 = 0\\
              \gamma_2^2 = \gamma_2'^2 = 0, \gamma_2\gamma_2' = 4\delta_4
          \end{pmatrix}}, & \text{if $i=3$;}
  \end{cases}
\]
where $\dim(\gamma_2) = \dim(\gamma_2') = 2$, $\dim(\delta_4) = 4$.
\end{theorem}

Let's consider the extension
\[
\begin{tikzcd}
  \phantom{.}1 \arrow{r} & \zz_b \arrow{r} & \zz_b\times Q_{2^i} \arrow{r} & Q_{2^i} \arrow{r} & 1
\end{tikzcd}
\]
and the associated spectral sequence $E$ with
\[
E_2^{p,q} = H^p(Q_{2^i}; H^q(\zz_b;\zz)) \then H^{p+q}(\zz_b\times
Q_{2^i}; \zz).
\]
The group $Q_{2^i}$ acts trivially on $\zz_b$, hence $Q_{2^i}$ also
acts trivially on $H^*(\zz_b;\zz)$. It follows now from~\cite[section
  XII.7]{cartan} and the fact that $b$ is odd that
\[
E_2^{p,q} = \begin{cases}
\zz, & \text{if $p=q=0$;} \\
\zz_b & \text{if $p=0$ and $q\ge 2$ is even;} \\
\zz_2^2, & \text{if $q=0$ and $p\equiv 2\pmod{4}$;} \\
\zz_{2^i}, & \text{if $q=0$ and $p\equiv 0\pmod{4}$ for $p\ge 4$;} \\
0, & \text{otherwise.}
\end{cases}
\]

Therefore $E_\infty^{p,q} = E_2^{p,q}$, we have no extension problems
to consider since $b$ is odd, and the induced product on $E_2$ gives
us the cohomology ring of $\zz_b\times Q_{2^i}$: writing
$H^*(\zz_b;\zz) = \dfrac{\zz[\beta_2]}{(b\beta_2=0)}$ with
$\dim(\beta_2) = 2$, we have $\beta_2\gamma_2 = \beta_2\gamma_2' =
\beta_2\delta_4 = 0$ in $H^*(\zz_b\times Q_{2^i}; \zz)$, plus the
relations among $\gamma_2$, $\gamma_2'$ and $\delta_4$ that already
hold in $H^*(Q_{2^i};\zz)$. Thus we have proved the following:

\begin{theorem}\label{theorem:ringzbq2i}
The cohomology ring $H^*(\zz_b\times Q_{2^i};\zz)$ has the following
presentation for $i\ge 3$:
\[
H^*(\zz_b\times Q_{2^i}; \zz) = \begin{cases}
  \dfrac{\zz[\beta_2,\gamma_2,\gamma_2',\delta_4]}
        {\begin{pmatrix}
            b\beta_2 = 2\gamma_2 = 2\gamma_2' = 2^i\delta_4 = 0\\
            \beta_2\gamma_2 = \beta_2\gamma_2' = \beta_2\delta_4 = 0\\
            \gamma_2^2 = 0, \gamma_2'^2 = \gamma_2\gamma_2' = 2^{i-1}\delta_4
        \end{pmatrix}}, & \text{if $i>3$;}\\
        \dfrac{\zz[\gamma_2,\gamma_2',\delta_4]}
              {\begin{pmatrix}
                  b\beta_2 = 2\gamma_2 = 2\gamma_2' = 8\delta_4 = 0\\
                  \beta_2\gamma_2 = \beta_2\gamma_2' = \beta_2\delta_4 = 0\\
                  \gamma_2^2 = \gamma_2'^2 = 0, \gamma_2\gamma_2' = 4\delta_4
              \end{pmatrix}}, & \text{if $i=3$;}
\end{cases}
\]
where $\dim(\beta_2) = \dim(\gamma_2) = \dim(\gamma_2') = 2$,
$\dim(\delta_4) = 4$.
\end{theorem}

The next step is to consider the extension
\[
\begin{tikzcd}
\phantom{,}1 \arrow{r} & \zz_a \arrow{r} & \zz_a\rtimes_\beta
(\zz_b\times Q_{2^i}) \arrow{r} & \zz_b\times Q_{2^i} \arrow{r} & 1,
\end{tikzcd}
\]
where $\beta\colon \zz_b\times Q_{2^i} \to \Aut(\zz_a)$ is given by
\begin{equation}\label{eq:definebeta}
\begin{array}{l}
  \beta(1_b)(1_a) = r\cdot 1_a,\\
  \beta(1_b)(x) = r_x\cdot 1_a,\\
  \beta(1_b)(y) = r_y\cdot 1_a,
\end{array}
\end{equation}
with $r$, $r_x$ and $r_y$ satisfying $r^b\equiv r_x^2\equiv r_y^2
\pmod{a}$. The Lyndon-Hochschild-Serre sequence associated to the
above extension is such that
\[
E_2^{p,q} = H^p(\zz_b\times Q_{2^i}; H^q(\zz_a;\zz)) \then
H^{p+q}(\zz_a\rtimes_\beta(\zz_b\times Q_{2^i});\zz).
\]
Since the ring $H^*(\zz_a;\zz)$ is generated by an element of
dimension $2$, we have
\begin{equation}\label{eq:e2zazbq2i}
E_2^{p,q} = \begin{cases}
  H^p(\zz_b\times Q_{2^i};\zz), & \text{if $q=0$;} \\
  H^p(\zz_b\times Q_{2^i}; \tilde{\zz}_a), & \text{if $q\ge 2$ is even;} \\
  0, & \text{if $q$ is odd;}
\end{cases}
\end{equation}
where $\tilde{\zz}_a = H^q(\zz_a;\zz)$ represents the $(\zz_b\times
Q_{2^i})$-module $\zz_a$ with the appropriate action of $(\zz_b\times
Q_{2^i})$. For $j>0$ and the given $\beta\colon \zz_b\times Q_{2^i}\to
\Aut(\zz_a)$, the action of $1_b\in\zz_b$ on $H^{2j}(\zz_a;\zz) =
\tilde{\zz}_a$ is the multiplication by $r^j$, the action of $x\in
Q_{2^i}$ is the multiplication by $r_x^j$ and the action of $y\in
Q_{2^i}$ is the multiplication by $r_y^j$. If we now consider the
spectral sequence $\tilde{E}$ associated to the extension
\[
\begin{tikzcd}
  1 \arrow{r} & \zz_b \arrow{r} & \zz_b\times Q_{2^i} \arrow{r} & Q_{2^i} \arrow{r} & 1
\end{tikzcd}
\]
with coefficients in $\tilde{\zz}_a$, we have
\[
\tilde{E}_2^{p,q} = H^p(Q_{2^i}; H^q(\zz_b;\tilde{\zz}_a)) \then
H^{p+q}(\zz_b\times Q_{2^i}; \tilde{\zz}_a),
\]
and from this we actually get
\[
\tilde{E}_2^{p,q} = \begin{cases}
  H^{p}(Q_{2^i};\zz_{\delta_j}), & \text{if $q=0$;} \\
  0, & \text{if $q>0$;}
\end{cases}
\]
with $\delta_j = \gcd(a,r^j-1)$. The spectral sequence therefore
collapses on $\tilde{E}_2$ and, since $\gcd(\delta_j, 2^i)=1$, we have
\begin{equation}\label{eq:zbq2izatilde}
H^p(\zz_b\times Q_{2^i}; \tilde{\zz}_a) = H^p(Q_{2^i}; \zz_{\delta_j}) = \begin{cases}
  \zz_{\varepsilon_j}, & \text{if $p=0$;} \\
  0, & \text{if $p>0$;}
\end{cases}
\end{equation}
where $\varepsilon_j = \gcd(a, r^j-1, r_x^j-1, r_y^j-1)$. Note that
$\varepsilon_j = \delta_j$ when $j$ is even, since $r_x^2\equiv
r_y^2\equiv 1 \pmod{a}$.

Plugging these results back into~(\ref{eq:e2zazbq2i}), we get
$E_\infty^{p,q} = E_2^{p,q}$ and, since $\gcd(a,b) = \gcd(ab,2)=1$,
the following result is then proved:

\begin{theorem}
The integral cohomology groups of $\zz_a\rtimes_\beta(\zz_b\times
Q_{2^i})$, for $\beta\colon \zz_b\times Q_{2^i} \to \Aut(\zz_a)$
defined by
\[
\begin{array}{l}
  \beta(1_b)(1_a) = r\cdot 1_a,\\
  \beta(1_b)(x) = r_x\cdot 1_a,\\
  \beta(1_b)(y) = r_y\cdot 1_a,
\end{array}
\]
are given by
\[
H^n(\zz_a\rtimes_\beta(\zz_b\times Q_{2^i}); \zz) = \begin{cases}
\zz, & \text{if $n=0$;} \\
0, & \text{if $n$ is odd;} \\
\zz_{\varepsilon_{2j+1}}\oplus\zz_b\oplus\zz_2^2, & \text{if $n = 4j+2$;} \\
\zz_{\delta_{2j}}\oplus\zz_b\oplus\zz_{2^i}, & \text{if $n = 4j$ (and $j>0$).} \\
\end{cases}
\]
\end{theorem}

Our work above actually allows us to determine the integral cohomology
ring of $\zz_a\rtimes_\beta(\zz_b\times Q_{2^i})$, and not just its
cohomology groups.

\begin{theorem}\label{theorem:zazbq2iring}
The cohomology ring $H^*(\zz_a\rtimes_\beta(\zz_b\times
Q_{2^i});\zz)$ is generated by the elements
\[
\left(\frac{a}{\varepsilon_j}\right)\alpha_2^j, \: \beta_2, \:
\gamma_2, \: \gamma_2', \text{ and } \delta_4,
\]
where $j>0$, $\varepsilon_j = \gcd(a,r^j-1, r_x^j-1, r_y^j-1)$,
$\dim\left(\dfrac{a}{\varepsilon_j}\alpha_2^j\right) = 2j$,
$\dim(\beta_2) = \dim(\gamma_2) = \dim(\gamma_2') = 2$ and
$\dim(\delta_4) = 4$. We have
$\varepsilon_j\left(\dfrac{a}{\varepsilon_j}\alpha_2^j\right)=0$, the
relations among $\beta_2$, $\gamma_2$, $\gamma_2'$ and $\delta_4$ are
exactly the ones described in Theorem~\ref{theorem:ringzbq2i}, and
\begin{gather*}
\left(\frac{a}{\varepsilon_i}\alpha_2^i\right) \smile
\left(\frac{a}{\varepsilon_j}\alpha_2^j\right) =
\frac{a\varepsilon_{i+j}}{\varepsilon_i\varepsilon_j}\left(\frac{a}{\varepsilon_{i+j}}\alpha_2^{i+j}\right),\\
\left(\frac{a}{\varepsilon_i}\alpha_2^i\right) \smile \beta_2 =
\left(\frac{a}{\varepsilon_i}\alpha_2^i\right) \smile \gamma_2 =
\left(\frac{a}{\varepsilon_i}\alpha_2^i\right) \smile \gamma_2' =
\left(\frac{a}{\varepsilon_i}\alpha_2^i\right) \smile \delta_4 = 0.
\end{gather*}
Moreover, let $d$ be the order of $r$ in $\U(\zz_a)$. If $d$ is even,
\[
(\alpha_2^d+\beta_2^d+\delta_4^{d/2})\smile\underline{\phantom{M}}\colon
H^{n}(\zz_a\rtimes_\beta(\zz_b\times Q_{2^i});\zz) \to
H^{n+2d}(\zz_a\rtimes_\beta(\zz_b\times Q_{2^i});\zz)
\]
is an isomorphism for $n>0$. If $d$ is odd, then
\[
(\alpha_2^{2d}+\beta_2^{2d}+\delta_4^d)\smile\underline{\phantom{M}}\colon
H^{n}(\zz_a\rtimes_\beta(\zz_b\times Q_{2^i});\zz) \to
H^{n+4d}(\zz_a\rtimes_\beta(\zz_b\times Q_{2^i});\zz)
\]
is an isomorphism for $n>0$.
\end{theorem}

\begin{proof}
We have just shown that the spectral sequence with $E_2$ page given
by~(\ref{eq:e2zazbq2i}) is such that $E_2^{p,q} = E_\infty^{p,q}$ and
$E_2^{p,q}=0$ if $p$ and $q$ are both positive and also if one of $p$
or $q$ is odd. Hence the product on $E_2$ induces the cup product on
$H^*(\zz_a\rtimes_\beta(\zz_b\times Q_{2^i}); \zz)$. The derivation of
equation~(\ref{eq:zbq2izatilde}) shows that one generator for
$E_2^{0,2j} = \zz_{\varepsilon_j}$ is
$\left(\dfrac{a}{\varepsilon}\right)\alpha_2^j$, where $\alpha_2$ is
one generator of $H^2(\zz_a;\zz)$ such that $H^*(\zz_a;\zz) =
\dfrac{\zz[\alpha_2]}{(a\alpha_2=0)}$. If $d$, the order of $r$ in
$\U(\zz_a)$, is even, then $\alpha_2^d$ is one generator of
$E_2^{0,2d}$ and the cup product with $\alpha_2^d$ defines an
isomorphism
\[
\begin{tikzcd}
E_2^{0,q} \arrow{r}{\cong} & E_2^{0,2d+q}
\end{tikzcd}
\]
for all $q>0$. Also, the cup product with $\beta_2^2+\delta_4$ defines
an isomorphism
\[
\begin{tikzcd}
E_2^{p,0} \arrow{r}{\cong} & E_2^{p+4,0}
\end{tikzcd}
\]
for $p>0$, as can be seen from
Theorem~\ref{theorem:ringzbq2i}. Therefore, if $d$ is even, the cup
product with $(\alpha_2^d+\beta_2^d+\delta_4^{d/2})\in
H^{2d}(\zz_a\rtimes_\beta(\zz_b\times Q_{2^i});\zz)$ defines an
isomorphism
\[
\begin{tikzcd}
H^{n}(\zz_a\rtimes_\beta(\zz_b\times Q_{2^i});\zz) \arrow{r}{\cong} &
H^{n+2d}(\zz_a\rtimes_\beta(\zz_b\times Q_{2^i});\zz)
\end{tikzcd}
\]
for $n>0$. If $d$ is odd, then the cup product with
$(\alpha_2^{2d}+\beta_2^{2d}+\delta_4^d)\in
H^{4d}(\zz_a\rtimes_\beta(\zz_b\times Q_{2^i});\zz)$ defines an
isomorphism
\[
\begin{tikzcd}
H^{n}(\zz_a\rtimes_\beta(\zz_b\times Q_{2^i});\zz) \arrow{r}{\cong} &
H^{n+4d}(\zz_a\rtimes_\beta(\zz_b\times Q_{2^i});\zz)
\end{tikzcd}
\]
for $n>0$.
\end{proof}

\bigskip
Finally, let $G = [\zz_a\rtimes_\beta(\zz_b\times
  Q_{2^i})]\rtimes_\theta\zz$ for some $\theta\colon \zz \to \Aut(\zz_a\rtimes_\beta(\zz_b\times
Q_{2^i}))$ and let's consider the extension
\[
\begin{tikzcd}
\phantom{.}1 \arrow{r} & \zz_a\rtimes_\beta(\zz_b\times Q_{2^i}) \arrow{r} &
G \arrow{r} & \zz \arrow{r} & 1.
\end{tikzcd}
\]
Corollary~\ref{cor:FZ} says that
\[
H^n\bigl([\zz_a\rtimes_\beta(\zz_b\times Q_{2^i})]\rtimes_\theta\zz; \zz\bigr) =
\begin{cases}
H^n(\zz_a\rtimes_\beta(\zz_b\times Q_{2^i}); \zz)^\zz, & \text{if $n$ is even;} \\
H^{n-1}(\zz_a\rtimes_\beta(\zz_b\times Q_{2^i}); \zz)_\zz, & \text{if $n$ is odd;}
\end{cases}
\]
so we will study now the action of $\zz$ on
$H^*(\zz_a\rtimes_\beta(\zz_b\times Q_{2^i});\zz)$. The action on
$H^0(\zz_a\rtimes_\beta(\zz_b\times Q_{2^i}); \zz)$ is trivial, hence
$H^n([\zz_a\rtimes_\beta(\zz_b\times Q_{2^i})]\rtimes_\theta\zz; \zz) = \zz$
for $n=0$ and $n=1$. For $j>0$, we have

\begin{align*}
\Aut(\zz_{\varepsilon_j}\oplus\zz_b\oplus\zz_2^2) &\cong \Aut(\zz_{\varepsilon_j})\oplus \Aut(\zz_b)\oplus \Aut(\zz_2^2),\\
\Aut(\zz_{\varepsilon_j}\oplus\zz_b\oplus\zz_{2^i}) &\cong \Aut(\zz_{\varepsilon_j})\oplus \Aut(\zz_b)\oplus \Aut(\zz_{2^i}),\\
\end{align*}
since $\varepsilon_j$, $b$ and $2$ are pairwise coprime. Therefore, if
$j>0$, in order to study the action of $\zz$ on
\[
H^{2j}(\zz_a\rtimes_\beta(\zz_b\times Q_{2^i}); \zz) =
\begin{cases}
\zz_{\varepsilon_j}\oplus\zz_b\oplus\zz_2^2, & \text{if $j$ is odd,} \\
\zz_{\varepsilon_j}\oplus\zz_b\oplus\zz_{2^i}, & \text{if $j$ is even,}
\end{cases}
\]
we must study the $\zz$-action on each of the summands
\begin{align}\label{eq:summands}
\zz_{\varepsilon_j} &\cong H^0(Q_{2^i}; H^0(\zz_b; H^{2j}(\zz_a;\zz))), \notag\\
\zz_b &\cong H^0(Q_{2^i}; H^{2j}(\zz_b; H^0(\zz_a; \zz))), \\
\zz_2^2 & \cong H^{2j}(Q_{2^i}; H^0(\zz_b; H^0(\zz_a; \zz))), & \text{($j$ odd)} \notag\\
\zz_{2^i} & \cong H^{2j}(Q_{2^i}; H^0(\zz_b; H^0(\zz_a; \zz))) & \text{($j$ even)} \notag
\end{align}

Let
$\theta\colon \zz \to \Aut(\zz_a\rtimes_\beta(\zz_b\times Q_{2^i}))$
be defined by
\begin{align}\label{eq:definetheta}
\theta(1)(1_a) &= c_a\cdot 1_a,\notag\\
\theta(1)(1_b) &= c\cdot 1_a + c_b\cdot 1_b,\\
\theta(1)(x) &= c_x\cdot 1_a + x^k, & \text{($k$ odd)} \notag\\
\theta(1)(y) &= c_y\cdot 1_a + x^{\ell}y, & \text{($\ell\in\{0,1,\ldots,2^{i-1}-1\}$)} \notag
\end{align}
where $c$, $c_a$, $c_b$, $c_x$ and $c_y$ are chosen according
to~\cite{DacibergMarek3}. Let's denote by $\theta_a$ the map
$\theta(1)|_{\zz_a}\colon \zz_a\to \zz_a$, by $\theta_b$ the map
$\overline{\theta(1)}|_{\zz_b}\colon \zz_b\to\zz_b$, where
$\overline{\theta(1)}\colon \zz_b\times Q_{2^i} \to \zz_b\times
Q_{2^i}$ is the induced map on the quotient, and by $\theta_Q$ the map
$\overline{\overline{\theta(1)}}\colon Q_{2^i} \to Q_{2^i}$. With this
notation, we have
\[
\theta_a(1_a) = c_a\cdot 1_a, \quad \theta_b(1_b) = c_b\cdot 1_b, \quad
\theta_Q(x) = x^k, \quad \theta_Q(y) = x^\ell y.
\]
Now the map $H^*(\theta_Q; H^*(\theta_b; H^*(\theta_a;\zz)))$ can be
used to calculate the $\zz$-action on each of the summands described
in Equation~(\ref{eq:summands}). We get that the $\zz$-action on
$H^0(Q_{2^i}; H^0(\zz_b;H^{2j}(\zz_a;\zz)))$ is the multiplication by
$c_a^j$ on $\zz_{\varepsilon_j}$ and the $\zz$-action on $H^0(Q_{2^i};
H^{2j}(\zz_b; H^0(\zz_a; \zz)))$ is the multiplication by $c_b^j$ on
$\zz_b$. As $\Aut(Q_8)$ doesn't follow the pattern of $\Aut(Q_{2^i})$
for $i>3$, the analysis of the $\zz$-action on $H^0(Q_{2^i};
H^{2j}(\zz_b; H^0(\zz_a; \zz)))$ will be broken in the cases $i=3$ and
$i>3$.

\begin{lemma}\label{lemma:theta2}
Let $i>3$ and let $\theta$ be the $\zz$-action on $Q_{2^i}$ given by
$\theta(1)(x)=x^k$, with $k$ odd, and $\theta(1)(y)=x^\ell y$, with
$\ell\in\{0,1,\ldots,2^{i-1}-1\}$. The induced $\zz$-action
$\theta^{(2)}$ on $H^2(Q_{2^i};\zz)$ is the trivial action if $\ell$
is even. If $\ell$ is odd, $\theta^{(2)}$ is determined by the
automorphism $\theta^{(2)}(1)$ given by
$\theta^{(2)}(1)(\gamma_2)=\gamma_2$,
$\theta^{(2)}(1)(\gamma_2')=\gamma_2+\gamma_2'$, where $\gamma_2$ and
$\gamma_2'$ are the generators of $H^*(Q_{2^i};\zz)$ given by
Theorem~\ref{theorem:TomodaPeter}.
\end{lemma}
\begin{proof}
We have $H^2(Q_{2^i}; \zz)\cong H_1(Q_{2^i}; \zz) \cong (Q_{2^i})_{ab}
\cong \zz_2\oplus\zz_2$, with generators $\gamma_2=\overline{x}$ e
$\gamma_2'=\overline{y}$. Since $\theta(1)(x)=x^k$ and $k$ is odd,
$\theta^{(2)}(1)(\alpha_2)=\alpha_2$, and $\theta(1)(y)=x^\ell y$
implies $\theta^{(2)}(1)(\gamma_2')=\gamma_2+\gamma_2'$, if $\ell$ is
odd and $\theta^2(1)(\gamma_2')=\gamma_2'$ if $\ell$ is even.
\end{proof}

By \cite{DacibergMarek3} and \cite{Swan}, the induced $\zz$-action on
$H^{2j}(Q_{2^i};H^0(\zz_b;H^0(\zz_a;\zz)))$ for $j>0$ even is the
multiplication by $k^j$. From that, the following result is obtained:

\begin{proposition}\label{proposition:theta2=theta4m+2}
Let $\theta$ be the $\zz$-action on $Q_{2^i}$ of the previous
Lemma. The $\zz$-action $\theta^{(4m+2)}$ induced by $\theta$ on
$H^{4m+2}(Q_{2^i}; \zz)$ coincides with $\theta^{(2)}$ induced on
$H^2(Q_{2^i};\zz)$.
\begin{proof}
By Theorem~\ref{theorem:TomodaPeter}, the ring $H^*(Q_{2^i};\zz)$
is given by
\[
\frac{\zz[\gamma_2,\gamma_2',\delta_4]}
     {(2\gamma_2=2\gamma_2'=2^i\delta_4=0,\gamma_2^2=0,{\gamma_2'}^{2}=\gamma_2\gamma_2'=2^{i-1}\delta_4)}.
\]
The generators of $H^{4m+2}(Q_{2^i}; \zz)$ are $\delta_4^m\gamma_2$
and $\delta_4^m\gamma_2'$. If the induced $\zz$-action on
$H^2(Q_{2^i};\zz)$ is the trivial one, it is clear that the induced
action on $H^{4m+2}(Q_{2^i};\zz)$ is also trivial. If the action on
$H^2(Q_{2^i}; \zz)$ is not trivial, then
\[
\theta^{(4m+2)}(1)(\delta_4^m\gamma_2)=k^{2m}\delta_4^m\gamma_2
=\delta_4^mk^{2m}\gamma_2 = \delta_4^m\gamma_2
\]
and
\[
\theta^{(4m+2)}(1)(\delta_4^m\gamma_2')=k^{2m}\delta_4^m\gamma_2 + k^{2m}\delta_4^m\gamma_2'
=\delta_4^mk^{2m}\gamma_2+\delta_4^mk^{2m}\gamma_2' = \delta_4^m\gamma_2 + \delta_4^m\gamma_2',
\]
since $k^{2m}$ is odd, and the result follows.
\end{proof}
\end{proposition}

\begin{lemma}
For $Q_8= \langle x,y \mid x^2=y^2, \: xyx=y\rangle$ any action of
$\zz$ on $H^2(Q_8;\zz )=\zz_2\oplus\zz_2$ can be realized as an
induced action of $\zz$ on $Q_8$.
\end{lemma}
\begin{proof}
As in Lemma~\ref{lemma:theta2}, the generators of
$H^2(Q_8)=\zz_2\oplus\zz_2$ are $\gamma_2=\overline{x}$ and
$\gamma_2'=\overline{y}$. Let
$\begin{bmatrix}\overline{a}&\overline{b}
  \\ \overline{c}&\overline{d}\end{bmatrix} \in M_2(\zz_2)$ be the
matrix that represents some $\zz$-action on $H^2(Q_8;\zz) =
\zz_2\oplus\zz_2$, where $a$, $b$, $c$, $d$ are integers and
$\overline{a}$, $\overline{b}$, $\overline{c}$, $\overline{d}$
represent classes modulo $2$. One action of $\zz$ on $Q_8$ that induces
the desired action on $H^2(Q_8;\zz)$ is $\theta(1)(x)=x^{a}y^{c}$ and
$\theta(1)(y)=x^{b}y^{d}$.
\end{proof}

Since any action of $\zz$ on $Q_8$ induces the trivial action on
$H^4(Q_8;\zz)$, a result analogue to
Proposition~\ref{proposition:theta2=theta4m+2} holds:

\begin{proposition}\label{theta2=theta4m+2}
Let $\theta$ be a $\zz$-action on $Q_8$. The induced $\zz$-action
$\theta^{(4m+2)}$ on $H^{4m+2}(Q_8; \zz)$ coincides with the induced
action $\theta^{(2)}$ on $H^2(Q_8; \zz)$.\qed
\end{proposition}

Let $A_j=\gcd(c_a^j-1,\varepsilon_j)$, $B_j=\gcd(c_b^j-1,b)$ e
$C_j=\gcd(k^j-1,2^i)$, and recall that $\theta_Q(x)=x^k$ and
$\theta_Q(y)=x^\ell y$. From our work so far in this section we can
compute the cohomology groups
$H^*\bigl([\zz_a\rtimes_\beta(\zz_b\times Q_{2^i})]\rtimes_\theta\zz;
\zz\bigr)$.

\begin{theorem}\label{theorem:ggroups}
Let $G = [\zz_a\rtimes_\beta(\zz_b\times Q_{2^i})]\rtimes_\theta\zz$,
where $\gcd(a,b) = \gcd(ab,2) = 1$, $i\ge 3$, and $\beta$ and $\theta$
are given by equations~(\ref{eq:definebeta})
and~(\ref{eq:definetheta}), respectively. If\/ $i>3$ and $\ell$ is
even, then
\[
H^n(G;\zz) =
\begin{cases}
\zz, & \text{if $n=0$ or $n=1$,}\\
\zz_{A_{2j}}\oplus\zz_{B_{2j}}\oplus\zz_{C_{2j}}, & \text{if $n=4j$ or $n=4j+1$ (and $j>0$),}\\
\zz_{A_{2j+1}}\oplus\zz_{B_{2j+1}}\oplus\zz_2^2, & \text{if $n=4j+2$ or $n=4j+3$.}\\
\end{cases}
\]
If\/ $i>3$ and $\ell$ is odd, then
\[
H^n(G;\zz) =
\begin{cases}
\zz, & \text{if $n=0$ or $n=1$,}\\
\zz_{A_{2j}}\oplus\zz_{B_{2j}}\oplus\zz_{C_{2j}}, & \text{if $n=4j$ or $n=4j+1$ (and $j>0$),}\\
\zz_{A_{2j+1}}\oplus\zz_{B_{2j+1}}\oplus\zz_2, & \text{if $n=4j+2$ or $n=4j+3$.}\\
\end{cases}
\]
If\/ $i=3$ and the induced action $\theta^{(2)}$ on $H^2(Q_8;\zz)$ is
trivial, then
\[
H^n(G;\zz) =
\begin{cases}
\zz, & \text{if $n=0$ or $n=1$,}\\
\zz_{A_{2j}}\oplus\zz_{B_{2j}}\oplus\zz_8, & \text{if $n=4j$ or $n=4j+1$ (and $j>0$),}\\
\zz_{A_{2j+1}}\oplus\zz_{B_{2j+1}}\oplus\zz_2^2, & \text{if $n=4j+2$ or $n=4j+3$.}\\
\end{cases}
\]
If\/ $i=3$ and the induced action $\theta^{(2)}$ on $H^2(Q_8;\zz)$ is
non-trivial, then
\[
H^n(G;\zz) =
\begin{cases}
\zz, & \text{if $n=0$ or $n=1$,}\\
\zz_{A_{2j}}\oplus\zz_{B_{2j}}\oplus\zz_8, & \text{if $n=4j$ or $n=4j+1$ (and $j>0$),}\\
\zz_{A_{2j+1}}\oplus\zz_{B_{2j+1}}\oplus\zz_2, & \text{if $n=4j+2$ or $n=4j+3$.}\\
\end{cases}
\]
\end{theorem}

\bigskip

We finish by making some remarks about the cohomology ring
$H^*(G;\zz)$. Let $d_k$ be the order of $k$ in $\U(\zz_{2^i})$,
$d_{c_a}$ be the order of $c_a$ in $\U(\zz_a)$, $d_{c_b}$ be the order
of $c_b$ in $\U(\zz_b)$, and $d$ be the order of $r$ in
$\U(\zz_a)$. When $i>3$ and $\ell$ is even, generators for the cyclic
subgroups $\zz_{A_{2j}}$, $\zz_{B_{2j}}$ and $\zz_{C_{2j}}$ of
$H^{4j}(G;\zz)$ (for $j>0$) are given by
$\left(\dfrac{a}{A_{2j}}\right)\alpha_2^{2j}$,
$\left(\dfrac{b}{B_{2j}}\right)\beta_2^{2j}$ and
$\left(\dfrac{2^i}{C_{2j}}\right)\delta_4^j$, respectively. Similarly,
the generators for the subgroups $\zz_{A_{2j+1}}$, $\zz_{B_{2j+1}}$
and $\zz_{C_{2j+1}}$ of $H^{4j+2}(G;\zz)$ are given by
$\left(\dfrac{a}{A_{2j+1}}\right)\alpha_2^{2j+1}$,
$\left(\dfrac{b}{B_{2j+1}}\right)\beta_2^{2j+1}$ and
$\left(\dfrac{2^i}{C_{2j+1}}\right)\delta_4^{j+1}$,
respectively. Also, for $n=4j+1$ (when $j>0$) and $n=4j+3$, the
generators of $H^n(G;\zz)$ are just the classes of the elements that
generate the respective cohomology groups
$H^*(\zz_a\rtimes_\beta(\zz_b\times\zz_{Q_{2^i}});\zz)$ that are given
by Theorem~\ref{theorem:zazbq2iring}. We can apply a similar reasoning
to the one used to prove Theorem~\ref{theorem:zazbzring} and
Theorem~\ref{theorem:zazbq2iring} and write down all the cup products
in $H^*(G;\zz)$ if we so desire, and we can also compute the period in
$H^*(G;\zz)$: if $p=\lcm(d,d_{c_a},d_{c_b},d_k)$, then we have $A_{2j}
= a$, $B_{2j} = b$ and $C_{2j} = 2^i$ if, and only if, $p\mid 2j$. It
now follows from equation~(\ref{eq:diagonalz}) and
Theorem~\ref{theorem:zazbq2iring} that, if $p$ is even,
\[
(\alpha_2^p+\beta_2^p+\delta_4^{p/2}) \smile \underline{\phantom{M}}\colon
H^n(G;\zz)\to H^{n+2p}(G;\zz)
\]
is an isomorphism for $n\ge 2$, and, if $p$ is odd,
\[
(\alpha_2^{2p}+\beta_2^{2p}+\delta_4^p) \smile \underline{\phantom{M}}\colon
H^n(G;\zz)\to H^{n+4p}(G;\zz)
\]
is an isomorphism for $n\ge 2$. The other cases are similar, and we
get the following result:

\begin{theorem}\label{theorem:gring}
Let $G$ be as in Theorem~\ref{theorem:ggroups}, and let $d_k$ be the
order of $k$ in $\U(\zz_{2^i})$, $d_{c_a}$ be the order of $c_a$ in
$\U(\zz_a)$, $d_{c_b}$ be the order of $c_b$ in $\U(\zz_b)$, and $d$
be the order of $r$ in $\U(\zz_a)$. Finally, let $p =
\lcm(d,d_{c_a},d_{c_b},d_k)$. If $p$ is even,
\[
(\alpha_2^p+\beta_2^p+\delta_4^{p/2}) \smile \underline{\phantom{M}}\colon
H^n(G;\zz)\to H^{n+2p}(G;\zz)
\]
is an isomorphism for $n\ge 2$, and, if $p$ is odd,
\[
(\alpha_2^{2p}+\beta_2^{2p}+\delta_4^p) \smile \underline{\phantom{M}}\colon
H^n(G;\zz)\to H^{n+4p}(G;\zz)
\]
is an isomorphism for $n\ge 2$.
\end{theorem}

\begin{corollary}
Using the same notation of the previous theorem, if $p$ is even, then
\[
(\widehat{\alpha_2^p+\beta_2^p+\delta_4^{p/2}}) \smile
\underline{\phantom{M}}\colon \hat{H}^n(G;\zz)\to \hat{H}^{n+2p}(G;\zz)
\]
is an isomorphism for all $n\in \zz$, and, if $p$ is odd,
\[
(\widehat{\alpha_2^{2p}+\beta_2^{2p}+\delta_4^p}) \smile
\underline{\phantom{M}}\colon \hat{H}^n(G;\zz)\to\hat{H}^{n+4p}(G;\zz)
\]
is an isomorphism for all $n\in \zz$.

\end{corollary}

\bibliographystyle{plain}
\bibliography{referencias}
\nocite{*}
\end{document}